\documentclass[12pt]{amsart}
\usepackage[all]{xy}
\usepackage{amsfonts}
\usepackage{amssymb}

\usepackage{color}

\newtheorem{teo}{Theorem}
\newtheorem{lem}[teo]{Lemma}
\newtheorem{prop}[teo]{Proposition}

\newtheorem{cor}[teo]{Corollary}

\newtheorem{dfn}[teo]{Definition}

\newtheorem{ex}[teo]{Example}

\newcommand{\id}{\operatorname{id}}
\newcommand{\Ker}{\operatorname{Ker}}

\def\<{\langle}
\def\>{\rangle}

\def\f{{\varphi}}

\def\Ker{\mathop{\rm Ker}\nolimits}
\def\Im{\mathop{\rm Im}\nolimits}

\begin{document}

\title[Hilbert $C^*$-modules with Hilbert dual]{Hilbert $C^*$-modules with Hilbert dual and $C^*$-Fredholm operators}
\author{Vladimir Manuilov}
\author{Evgenij Troitsky}
\thanks{ The work is supported by the Russian Science Foundation under grant 21-11-00080}
\address{Moscow Center for Fundamental and Applied Mathematics, 
	Lomonosov Moscow State University, 119991 Moscow, Russia}
\email{manuilov@mech.math.msu.su}
\email{troitsky@mech.math.msu.su}

\begin{abstract}
We study Hilbert $C^*$-modules over a $C^*$-algebra $A$ for which the Banach $A$-dual module
carries a natural structure of Hilbert $A$-module. In this direction we prove that 
if $A$ is monotone complete, $M$ and $N$ are Hilbert $A$-modules,  $M$ is self-dual, and both 
$T:M\to N$ and its Banach $A$-dual $T':N'\to M'$ have trivial kernels and cokernels then $M\cong N'$.
With the help of this result, for a monotone complete $C^*$-algebra $A$, we prove that the index of any $A$-Fredholm operator
can be calculated as the difference of its kernel and cokernel as in the Hilbert space case.
\end{abstract}

\maketitle

\section{Introduction}

Hilbert $C^*$-modules were introduced by W. Paschke in \cite{Pas1} as a generalization of Hilbert spaces to modules over a $C^*$-algebra $A$ with an inner product taking values also in $A$. 

In this paper we deal with right Hilbert $C^*$-modules, i.e. $A$ acts on a module $M$ from the right, and the inner product satisfies
$\<m,n a\>=\<m,n\>a$, $\<ma,n\>=a^*\<m,n\>$, $m,n \in M$, $a \in A$.
The dual module $M'$ consists of all bounded anti-$A$-linear maps $f:M\to A$ (called functionals),
i.e., maps satisfying $f(ma)=a^*f(m)$. 
The dual module is a Banach $A$-module with the multiplication defined by $(fa)(m)=f(m)a$, but not necessarily a Hilbert $C^*$-module, and the map $m\mapsto\widehat m=\langle \cdot,m\rangle$ is an isometric inclusion $M\subset M'$.  
A module $M$ is \emph{self-dual} if this inclusion is an isomorphism. Self-dual modules have a lot of nice properties, e.g. \cite{Pas1,Frank1990} (and also \cite{MTBook}), but they are rare. Unlike the first dual module, the second dual one, $M''$, has a natural structure of a Hilbert $C^*$-module, and there is an isometric inclusion $M\subset M''\subset M'$. This is the end of the story, as $M'''$ is canonically isomorphic to $M'$ \cite[Theorem ~2.4]{Pas2} (see Theorem \ref{teo:pasc} below). 
A module $M$ is \emph{reflexive} if $M=M''$. Reflexive modules are not as nice as self-dual ones, but still share some nice properties of the latter, see e.g. \cite{Frank1990,FMT2010NY,FMT2010Studia}
(and also \cite{MTBook}). The last remaining equality, $M'=M''$, was discussed in \cite{Manuilov2022}, where such modules were called \emph{modules with Hilbert dual}.

Hilbert $C^*$-modules over $W^*$-algebras have some special properties that make them closer to Hilbert spaces. In particular, although they need not be self-dual, the inner product extends to the dual modules, making them Hilbert $W^*$-modules as well. The same is true for monotone complete $C^*$-algebras \cite{FrankMaNa}, so any Hilbert $C^*$-module over such $C^*$-algebra has Hilbert dual. 

\medskip

It is well known that if there is a bounded operator $T:H_1\to H_2$ between two Hilbert spaces with trivial kernel and cokernel then $H_1$ and $H_2$ are isomorphic. Although Hilbert spaces differ only by cardinality, the construction of an isomorphism uses a polar decomposition for $T$, which need not to be invertible. 

When we replace Hilbert spaces by Hilbert $C^*$-modules, the result becomes more difficult, and definitely not true in full generality. Our aim is to clarify the conditions that allow to prove that two Hilbert $C^*$-modules, $M$ and $N$, over a $C^*$-algebra $A$ are isomorphic when connected by a bounded operator $T:M\to N$. Our main result claims that if $A$ is monotone complete, $M$ self-dual, and both $T$ and its dual $T':N'\to M'$ have trivial kernels and cokernels then $M\cong N'$ (Theorem \ref{teo:corrected_ext}). As an application, we generalize a result on decomposition of $A$-Fredholm operators from \cite{FrTroFA} to the case when $A$ is monotone complete 
(Theorem \ref{teo:Fredh}) and correct the proof in the $W^*$ case. 

\medskip\noindent
\textbf{Acknowledgment:} The authors are indebted to Michael Frank and Stefan Ivkovic for valuable 
discussions and to an anonymous reviewer for careful reading of the first version of the paper 
and very helpful and friendly criticism.

\section{Preliminaries}
Definitions of Hilbert $C^*$-modules, adjointable and non-adjointable operators and their basic properties can be found in textbooks, e.g. \cite{Pas1,MTBook,Lance,ManuilovTroit2000JMS}.

Recall that the canonical inclusions for dual and double dual modules \cite{Pas2}.
For $m\in M$, define $\dot{m}\in M''$ by 
\begin{equation}\label{11111}
\dot{m}(f)=(f(m))^*, \quad f\in M'.
\end{equation}
The map $m\mapsto \dot{m}$ is an isometric module map of $M$ to $M''$. For $X\in M''$ define $\widetilde{X}\in M'$ by $\widetilde{X}(m)=X(\widehat{m})$, where $m\mapsto \widehat{m}$, $M\to M'$ was defined above. So, if we identify $M$ with its image under $m\mapsto \widehat{m}$, then $\widetilde{X}$ is the restriction of $X$ on $M$. Notice that 
\begin{equation}\label{eq:Pas_identif0}
(\dot{m})\!{\widetilde{\phantom{m}}}=\widehat{m} \quad\mbox{ for all } m\in M.
\end{equation}
It is clear that the map $X\mapsto \widetilde{X}$ 
is a module map of $M''$ into $M'$ and that $\|\widetilde{X}\|_{M'}\le \|X\|_{M''}$ for $X\in M''$. Moreover, this map is an isometry as it follows from the following proposition.
\begin{prop}[Prop. 2.1 in \cite{Pas2}]
Let $N$ be a submodule of $M'$ containing $\widehat{M}$. For any $f\in N'$, we have $\|f\|_{N'}=\|f|_{\widehat{M}}\|$, where
\begin{equation}\label{eq:norm_dual}
\|f\|_{N'}=\sup_{m\in N,\|m\|\leq 1}\|f(m)\|.
\end{equation}
\end{prop}
After making all
permissible identifications, we can then array the modules $M$, $M'$, and $M''$ as
\begin{equation}\label{eq:Pas_identif}
M\overset{D}{\subseteq} M''\overset{I} {\subseteq} M', \qquad D(m)=\dot{m}, \quad I(X)=\widetilde{X}, \qquad I\circ D (m) = \widehat{m}
\end{equation}
by (\ref{eq:Pas_identif0}).

In  \cite{Pas2} the inner product 
\begin{equation}\label{eq:pas_prod}
\<\cdot,\cdot\>'': M'' \times M'' \to A, \qquad \<X,Y\>''=Y(I(X)),
\end{equation}
is introduced and the following statement is proved.
\begin{teo}[Theorem 2.4 in \cite{Pas2}]\label{teo:pasc}
The map \emph{(\ref{eq:pas_prod})} is an $A$-valued inner product on $M''$. The norm obtained from this product coincides with the initial (operator) norm on $M''$. The map $I$ from \emph{(\ref{eq:Pas_identif})} is an isometry of $M''$ into $M'$. 
\end{teo}

\section{Hilbert $C^*$-modules with Hilbert dual}

\begin{dfn}[\cite{Manuilov2022}]\label{def}\rm
A Hilbert $C^*$-module $M$ over a $C^*$-algebra $A$ is called a \textit{module with a Hilbert dual} if the inner product $\langle\cdot,\cdot\rangle$ on $M$ extends to an inner product $\langle\cdot,\cdot\rangle'$ on $M'$ such that 
\begin{itemize}
\item
$f(m)=\langle \widehat m,f\rangle'$ for any $m\in M$ and any $f\in M'$ (in particular, $\langle \widehat m,\widehat n\rangle'=\widehat n(m)=\langle m,n\rangle$ for any $m,n\in M$);
\item
$M'$ is complete with respect to the norm determined by the inner product $\langle\cdot,\cdot\rangle'$.
\end{itemize}
\end{dfn}

\begin{prop}
$M$ is a Hilbert $C^*$-module with a Hilbert dual if and only if $I(M'')=M'$.
\end{prop}

\begin{proof}
If $I(M'')=M'$,  we can consider the induced inner product on  $M'$: $\<f,g\>':=\<I^{-1}(f),I^{-1}(g)\>''$, where $f,g\in M'$.
Let $X=I^{-1}(f)$. Then, by (\ref{eq:Pas_identif0}) and (\ref{eq:pas_prod}),
\begin{eqnarray*}
\<\widehat{m},f \>'&=&\<I^{-1}(\widehat{m}),X\>''=\<I^{-1} I(\dot{m}),X\>''\\\
&=&(\<X, \dot{m}\>'')^*=(\dot m(I(X)))^*=(\dot m(f))^*\\
&=&f(m).
\end{eqnarray*}

Conversely let $M$ be a Hilbert $C^*$-module with a Hilbert dual, and let $\langle\cdot,\cdot\rangle'$ be an extension of the inner product on $M$ to $M'$. Then, besides the standard (operator)  norm $\|f\|_{M'}$ on $M'$ (see (\ref{eq:norm_dual})),  there is also the norm $\|f\|'=\|\langle f,f\rangle'\|^{1/2}$, $f\in M'$. As $M'$ is complete with respect to both norms, and we have the following estimation:
by  Definition \ref{def} $\<f,\widehat{m}\>' \<\widehat{m},f\>' \le (\|\widehat{m}\|')^2 \<f,f\>'= (\|m\|)^2 \<f,f\>'$ (see e.g. \cite[Prop. 1.2.4]{MTBook}) and
\begin{eqnarray*}
(\|f\|')^2&=&\|\<f,f\>'\| \ge \Bigl(\sup_{m\in M,\: \|m\|\le 1} \|\<f,m\>'\|\Bigr)^2\\
&=& \Bigl(\sup_{m\in M,\: \|m\|\le 1} \|f(m)\|\Bigr)^2=\left(\|f\|_{M'}\right)^2.
\end{eqnarray*}
Thus $\mathrm{Id}: (M',\|.\|') \to (M',\|.\|_{M'})$ is an injective bounded epimorphism and the norms
are equivalent, i.e. there exists $c>0$ such that $\|f\|_{M'}\leq \|f\|'\leq c\|f\|_{M'}$ for any $f\in M'$. Given $f\in M'$, consider the map $g\mapsto \langle g,f\rangle'$, $M'\to A$. It is obviously anti-$A$-linear. As the two norms are equivalent, this map is bounded by $c \|f\|_{M'}$. Hence this map is a functional on $M'$, i.e. is an element of $M''$. This gives a bounded map $\Phi:M'\to M''$, $\|\Phi\|\le c^2$, because
$$ 
\|\Phi(f)(g)\|=\|\<g,f\>'\|\le \|g\|'\|f\|'\le c^2 \|g\|_{M'} \|f\|_{M'}. 
$$
Consider the composition $I\circ\Phi:M'\to M'$. Then, for any $f\in M'$, $m\in M$, 
$$
 I\circ\Phi(f)(m)=\Phi(f)(\widehat m)=\<\widehat{m},f\>'=f(m). 
$$
Thus, the map $I\circ\Phi$ is the identity map. Therefore, $I$ is an epimorphism and $I(M'')=M'$.
Also, since $I$ is an isometry, $\Phi$ as an isometric isomorphism and one can \emph{aposteriori} take $c=1$.
\end{proof}

 The end of the proof shows that all inner products on $M'$ for a Hilbert module with a Hilbert dual
generate the same norm, hence coincide  with the operator norm on $M'$  (see also \cite[Propostion 3]{Manuilov2022}). 

 It also follows that if $M$ has Hilbert dual then $M'$ (which is a Hilbert $C^*$-module in this case) is self-dual.

 Any self-dual Hilbert $C^*$-module is obviously a module with a Hilbert dual. It was shown in \cite{Pas1} that any Hilbert $C^*$-module over a $W^*$-algebra is a module with a Hilbert dual, and it was shown in \cite{Hamana1992,Lin1992Pacific,FrankMaNa} that the same holds for a slightly bigger class of $C^*$-algebras, namely for monotone complete ones. Moreover, Theorem 4.7 of \cite{FrankMaNa} 
(an immediate consequence of \cite{Hamana1992})
claims that any Hilbert $C^*$-module over a $C^*$-algebra $A$ has a Hilbert dual if and only if $A$ is monotone complete.  

There are also examples of Hilbert $C^*$-modules with Hilbert dual over $C^*$-algebras beyond the class of monotone complete ones.

\begin{ex}
\rm
 Let $X$ be a locally compact Hausdorff space, $M=C_0(X)$ the algebra of continuous functions on $X$ vanishing at infinity, and let $A=C_b (X)$ be its multiplier algebra, i.e. the algebra of all bounded continuous functions on $X$. Then we can consider $M$ as a Hilbert $C^*$-module over $A$.

Let us verify that $M'=C_b (X)$. The inclusion $C_b (X) \subseteq M'$ is evident. Let now $\f\in M'$. Consider a countable partition of
unity $\{f_i\}$ on $X$ consisting of real-valued functions with compact supports. Then, for any $a \in M$, one has $a=\sum_i a f_i$ uniformly.
Then $\f(a)=\sum_i \f(a f_i)=\sum_i\f(f_i a)=\sum_i a^* \f(f_i)$ uniformly. So, $\sum_i \f(f_i)$ is a multiplier of $C_0(X)$ corresponding to $\f$.

So, $M'=A$ is a Hilbert module over $A$. Then it is clear that $M''=M'$, because each $A$-functional on a unital $A$ is determined by its value at $1_A$.

Recall that a commutative algebra is monotone complete iff it is an AW*-algebra iff its spectrum is extremally disconnected. So if $X$ is not extremely disconnected (e.g. $X=(0,1)$) then  $C_0(X)$ provides an example of a Hilbert $C^*$-module with Hilbert dual, although the underlying $C^*$-algebra $C_b(X)$ is not monotone complete. 
\end{ex}

\section{Around polar decomposition in Hilbert $C^*$-modules}

 Proposition 2 of \cite{FrTroFA} claims that if $T:M\to N$ is an injective bounded module homomorphism of Hilbert $C^*$-modules over a $W^*$-algebra, $T(M)^{\perp\perp}=N$ and if $M$ is self-dual then $T(T^*T)^{-1/2}$ is a bounded module isomorphism $M\cong N$  (although $(T^*T)^{-1/2}$ may be unbounded). S. Ivkovic has informed us that there is a problem in the proof, thus urging us to revise this claim. Here is a counterexample.

\begin{ex}\rm
Let $A=l_\infty$ be the $W^*$-algebra of bounded sequences, $M=A$. We shall use the notation $\delta_n\in A$ for the sequence of zeroes with the only unit at the $n$-th place, $\delta_n=(0,...,0,1,0,...)$. Consider the following closed submodule $N\subset l_2(A)$ of all sequences of the form $x=(a_1,a_2,...)$ where each $a_n\in A$, $n\in\mathbb N$, satisfies $a_n=\delta_n a_n$. If this set is not closed then take the closure. 
 
Fix $\lambda_n=n^{-1}$, and set $T(a)=(\lambda_1\delta_1 a,\lambda_2\delta_2 a,...)\in N$, $a\in A$. Then $T:M\to N$ is injective, and $T(M)^\perp=0$ in $N$. Set $f=T(T^*T)^{-1/2}$, then $f(a)=(\delta_1 a,\delta_2 a,...)$.
The range of $f$ is greater than $N$ (e.g. it contains $f(1)=(\delta_1,\delta_2,...)$, which is not in $N$). Moreover, $M=A$ is not isometrically isomorphic to $N$.
\end{ex}
  
The reason for the error in the proof of Proposition 2 of \cite{FrTroFA} is that $f$ maps $M$ not to $N$, but to the dual module $N'$. The problem with the Lin's proof of the polar decomposition \cite{Lin} is similar: he uses the polar decomposition of an operator over the enveloping $W^*$-algebra, but the unitary in this decomposition maps into the dual module.
 
The easiest way to fix  Proposition 2 of \cite{FrTroFA} is to add the requirement of self-duality for $N$
(see Corollary \ref{teo:corrected} below).

We want to pursue this issue further and to remove the requirement of self-duality for $N$.
 
Let $T:K\to M$ be an operator between two Hilbert $C^*$-modules. Define $T^\#:M\to K'$ by $(T^\#m)(k)=\langle Tk,m\rangle$, $m\in M$, $k\in K$. If $K\subset M$ then we denote by $J$ the corresponding inclusion, $J:K\to M$.

\begin{lem}\label{lem:inner}	
Suppose that $K\subset M$, $K^\perp=0$, $M$ is self-dual and $K$ has a Hilbert dual. Then $\|\langle J^\#m,J^\#m\rangle'\|=\|\langle m,m\rangle\|$ for any $m\in M$.  Also $J^\#:M\to K'$ is an isomorphism.
\end{lem}
\begin{proof}
The map $J^\#$ is given by $(J^\#m)(k)=\langle J(k),m\rangle$. Note that $J^\#$ is a contraction. Indeed, 
\begin{eqnarray*}
\|J^\#(m)\|&=&\sup_{k\in K,\|k\|\leq 1}\|J^\#(m)(k)\|\\
&=&\sup_{k\in K,\|k\|\leq 1}\|\langle J(k),m\rangle\|\leq\|m\|. 
\end{eqnarray*}
Thus, 
\begin{equation}\label{eq:estim_J_dies}
\| \langle J^\#m,J^\#m\rangle' \| \leq \| \langle m,m\rangle \|
\end{equation}
for any $m\in M$.

For $k,k'\in K$, we have
$J^\#(J(k))(k')=\<J(k'),J(k)\>=\<k',k\>$, i.e., $J^\#(J(k))=\widehat{k}$.

Let us define a map $Q:K'\to M$. Let $\tau\in K'$. Then the map 
$$
m\mapsto\bar\tau(m)=\langle J^\#m, \tau \rangle'
$$ 
is a functional on $M$ (algebraic properties are obvious, and boundedness follows from that of $J^\#$). Then self-duality of $M$ implies existence of a unique $m_\tau\in M$ such that  $\bar\tau(m)=\langle m,m_\tau\rangle$  for any $m\in M$. Set $Q(\tau)=m_\tau$. 

Then $\bar\tau(m_\tau)=\langle m_\tau,m_\tau\rangle$, 
and by (\ref{eq:estim_J_dies}) one has
\begin{eqnarray*}
\|m_\tau\|^2&=&\|\langle m_\tau,m_\tau\rangle\|=\|\langle\tau,J^\#(m_\tau)\rangle'\|\\
&\leq&\|\tau\| \cdot \|\langle J^\#(m_\tau),J^\#(m_\tau)\rangle'\|^{1/2}\\
&\leq& \|\tau\|\cdot \|\langle m_\tau,m_\tau\rangle\|^{1/2}\\
&=&\|\tau\|\cdot \|m_\tau\|,
\end{eqnarray*}
whence $\|m_\tau\|\leq\|\tau\|$, and $\|Q\|\leq 1$.

Consider the composition $Q\circ J^\#:M\to M$. Set $\tau=J^\#(m)$, $n=m_\tau$. Then $Q\circ J^\#(m)=n$, and
$$
\langle k,n\rangle=\bar\tau(k)=\tau(k)=\langle k,m\rangle
$$
for any $k\in K$. As $K^\perp=0$, we have $n=m$, i.e. $Q\circ J^\#=\operatorname{id}_M$. Since both $Q$ and $J^\#$ are contractions, $J^\#$ is an isometric embedding.

Taking the composition in the opposite order, we have, for any $\tau\in K'$ and $k \in K$,
\begin{eqnarray*}
(J^\#\circ Q (\tau) )(k) &=& (J^\#(m_\tau) )(k) = \<J(k),m_\tau\>\\
&=&\overline{\tau}(J(k))=\<J^\# J(k),\tau\>'\\
&=& \<\widehat{k},\tau\>'=
\tau(k) 
\end{eqnarray*}
and $J^\#\circ Q=\operatorname{id}_{K'}$. Hence, as above, $Q$  is an isometric embedding. Thus, $Q$ and $J^\#$ are isometric isomorphisms.
\end{proof}

\begin{cor}\label{ravenstvo}
The map $J^\#$ satisfies $\langle J^\#m,J^\#m\rangle'=\langle m,m\rangle$ for any $m\in M$.
\end{cor}

\begin{proof}
By Lemma \ref{lem:inner} and its proof,  $J^\#$ is an isometry with the inverse isometry $Q$.
Hence,
\begin{eqnarray*}
\<m,m\>&=&\<Q J^\# m, Q J^\#m\> \le \|Q\|^2 \<J^\# m,J^\# m\>'\\
&=&\<J^\# m,J^\# m\>' \le \|J^\#\|^2 \<m,m\>\\
&=&\<m,m\>.
\end{eqnarray*}
\end{proof}

\begin{teo}\label{teo:corrected_ext}
	Let $A$ be a monotone complete $C^*$-algebra, let $M$, $N$ be Hilbert $C^*$-modules over $A$, $M$ self-dual, and let $T:M\to N$ be an operator with $(T(M))^{\perp\perp}=N$, $\operatorname{Ker}(T)=\{0\}$ and $\operatorname{Ker}(T')=\{0\}$,
where $T':N'\to M'$ is the adjoint operator. 
	Then $M \cong N'$.
\end{teo}	

\begin{proof}
We have that $T':N'\to M'$ is  a morphism of self-dual modules (see Theorem \ref{teo:pasc})
with $\operatorname{Ker}(T')=\{0\}$. 
Since $M$ is self-dual, $T$ is adjointable \cite{Pas1}. 
One has $0=\<T^*(n),m\>=\<n, T(m)\>$ and 
$\operatorname{Ker}(T)=\{0\}$ implies $(T^*(N))^\bot =\{0\}$.

Denote $\Lambda_M: M \to M'$, $m \mapsto \widehat{m}$, and $\Lambda_N: N \to N'$, $n\mapsto \widehat{n}$. Then $T^*$ and $T'$ are connected by the following
commutative diagram:
\begin{equation}\label{eq:commut_diag_T'_T*}
\xymatrix{
M' & \ar[l]_{T'} N'\\
M \ar[u]^\cong_{\Lambda_M}& \ar[l]^{T^*} N \ar[u]_{\Lambda_N}.
}
\end{equation}
Thus $(T^*(N))^\bot =\{0\}$ implies $(T'(N'))^\bot =\{0\}$.

Since $N'$ is self-dual, $T'$ is adjointable \cite{Pas1}.
Let $S=((T')^*T')^{1/2}:N'\to N'$. Then $S(N')^\perp=\{0\}$. 
Indeed, if $f\in S(N')^\perp_{N'}$, then for $f'=S(f)$, one has 
$$
0=\<S(f'),f\>=\<(T')^*T' f,f\>=\<T'f,T'f\>
$$
and $f=0$. 
Denote the norm closure of $S(N')$ by $K$.  If $J$ denotes the inclusion of $K$ in $N'$ then $J^\#:N'\to K'$ is the above isometric isomorphism
(since $A$ is monotone complete, all modules are with a Hilbert dual).
We can apply Lemmas  \ref{lem:inner} and \ref{ravenstvo} 
because $N'$ is self-dual and $K^\bot \subseteq S(N')^\perp=\{0\}$.

Since $\<T'f,T'f\>=\<Sf,Sf\>$,
the map $U_0:S(N')\to M'$ defined by $U_0(Sf):=T'f$ is an isometric embedding and extends by continuity to an isometric embedding $U:K\to M'$ (because $M'$ is complete). 
Note that 
$$
U(K)^\bot \subseteq (U_0(S(N')))^\bot =(T'(N'))^\bot =\{0\}.
$$
So, since  $M'$ is self-dual, we can apply Lemmas 
\ref{lem:inner} and \ref{ravenstvo}. We obtain in this way an isometric isomorphism
$U^\#: M' \to K'$. Thus $(J^\#)^{-1}\circ U^\#: M'\cong K' \cong N'$ is the desired
isometric isomorphism.
\end{proof}	

\begin{cor}\label{teo:corrected}
	Let $A$ be a monotone complete $C^*$-algebra, let $M$, $N$ be self-dual Hilbert $C^*$-modules over $A$, and let $T:M\to N$ be an operator with $\operatorname{Ker}(T)=\{0\}$. 
	Suppose that $(T(M))^{\perp\perp}=N$.
	Then $M\cong N$.
\end{cor}	
	
\begin{proof}
In the self-dual case $\Lambda_N$ is an isomorphism and
$T' = \Lambda_M \circ T^* \circ (\Lambda_N)^{-1}$
(see (\ref{eq:commut_diag_T'_T*})).
If $T(M)^\bot =\{0\}$, then $\Ker(T^*)=\{0\}$. Indeed,
suppose that $T^*(n)=0$, then $\<T^*n,m\>=0$ for any $m\in M$, so 
$\<n,Tm\>=0$ for any $m\in M$, i.e. $n\bot T(M)$. Hence $n=0$ and $\Ker(T^*)=\{0\}$.
Thus, the property $\operatorname{Ker}(T')=\{0\}$ follows from the other conditions in this case.
\end{proof}	


\section{Decompositions into direct sums}
 
Here we explain that Corollary 3 in \cite{FrTroFA}, which was based on an erroneous statement, is nevertheless true.
As a bonus, we extend it from the case of Hilbert $C^*$-modules over $W^*$ algebras to that over monotone 
complete $C^*$-algebras.

The next statement is known for Hilbert $C^*$-modules over $W^*$-algebras (\cite{FrTroFA}, \cite[Lemma 3.6.1]{MTBook}).

\begin{lem}\label{lem:botbot_ker}
Let $L$ be a self-dual Hilbert $C^*$-module, and let $K\subseteq L$ be a Hilbert $C^*$-submodule with a Hilbert dual. Then we have the orthogonal decomposition:
$
L=K^{\bot\bot}\oplus K^\bot.
$
\end{lem}

\begin{proof}
For the inclusion $i:K\subseteq L$,
consider the Banach adjoint map $R=i': L' \to K'$ of self-dual Hilbert modules. The map $R$
is given by $R(f)(k)=f(i(k))$, $f\in L'$, $k\in K$, and 
satisfies the formula 
\begin{equation}\label{m1}
R(\widehat{l}\,)(k)=\widehat{l}(i(k))=\langle i(k),l\rangle_L=\langle \widehat{i(k)},\widehat{l}\rangle_{L'}, \quad l\in L, k\in K. 
\end{equation}
As both $K'$ and $L$ are self-dual, $R$
admits the Hilbert $C^*$-module adjoint morphism $R^*: K' \to L'\cong L$.

We claim that $R^*R$ is a projection, and that $\Ker(R)=\Ker(R^*R)$. 

The first claim would follow if we show that $RR^*=\id_{K'}$. Note that $R^*$ satisfies $\langle f,R(\widehat{l}\,)\rangle_{K'}=\langle R^*(f),\widehat{l}\,\rangle_{L'}$, $f\in K'$, $l\in L$. Taking here $f=\widehat{k}$, we get $\langle \widehat{k},R(\widehat{l}\,)\rangle_{K'}=\langle R^*(\widehat{k}),\widehat{l}\,\rangle_{L'}$. Compairing this with (\ref{m1}), we see that $R^*(\widehat{k})=\widehat{i(k)}$ for any $k\in K$. Next,
\begin{eqnarray*}
\langle RR^*(\widehat{k}),\widehat{r}\rangle_{K'}&=&\langle R^*(\widehat{k}),R^*(\widehat{r})\rangle_{L'}\\
&=&\langle \widehat{i(k)},\widehat{i(r)}\rangle_{L'}=\langle i(k),i(r)\rangle_{L}\\
&=&\langle k,r\rangle_{K}=\langle \widehat{k},\widehat{r}\rangle_{K'} 
\end{eqnarray*}
for any $k,r\in K$, hence $\langle RR^*(\widehat{k})-\widehat{k},\widehat{r}\rangle_{K'}=0$, which implies that $RR^*(\widehat{k})=\widehat{k}$ for any $k\in K$.  
Now let $f\in K'$. Then
$$
\langle RR^*(f),\widehat{k}\rangle_{K'}=\langle f,RR^*(\widehat{k})\rangle_{K'}=\langle f,\widehat{k}\rangle_{K'},
$$
or, alternatively, $RR^*(f)(k)=f(k)$ for any $k\in K$, hence $RR^*=\id_{K'}$. 

For the second claim, it is clear that $\Ker R\subset\Ker R^*R$. To prove the opposite inclusion suppose that $R(x)\neq 0$ for some $x\in L'=L$. Then $\<R(x),R(x)\>=\<R^*R(x),x\>\ne 0$, hence, $R^*Rx\ne 0$.


It remains to observe that $K^\bot = \Ker(R)$ and thus is orthogonally
complemented.
\end{proof}	

\begin{cor}\label{cor:self_du}
	Suppose that $M$ and $L$ are self-dual modules over a monotone complete
algebra, $T:M\to L$ is 
	an (adjointable) morphism with $\Ker(T)=\{0\}$. Then $M\cong (T(M))^{\bot\bot}$. 
\end{cor}	

\begin{proof}
Recall that all Hilbert $C^*$-modules over a monotone complete $C^*$-algebra have Hilbert dual. By Lemma \ref{lem:botbot_ker}, we have $L= (T(M))^{\bot\bot} \oplus (T(M))^{\bot}$. Set $N:= (T(M))^{\bot\bot}\supseteq T(M)$. Then $N= (T(M))^{\bot\bot}$ is a direct orthogonal summand of a self-dual module and thus a self-dual module itself. Consider $T$ as a map from $M$ to $N$ (we keep the same notation for this morphism). Then $T:M\to N$ satisfies the assumptions of Corollary \ref{teo:corrected}, whence $M\cong N'$. Self-duality of $N$ finishes the proof.
\end{proof}	

	Note that, for an adjointable operator $F:M\to N$, we have
	\begin{eqnarray*}
(\Im F)^{\bot}&=&\{y \in N \colon \<Fx,y\>=0 \mbox{ for all } x\in M\}\\
&=&
\{y \in N \colon \<x,F^*y\>=0 \mbox{ for all } x\in M\}\\
&=&\{y \in N \colon F^*y=0 \}\\
&=&\Ker F^*.
	\end{eqnarray*}
Hence,
\begin{equation}\label{eq:ker_and_coker}
(\overline{ \Im F})^{\bot}=(\Im F)^{\bot}=\Ker F^*\qquad\mbox{for an adjointable }F.
\end{equation}

From this observation and Lemma \ref{lem:botbot_ker} we obtain the following statement.

\begin{cor}\label{cor:K_botbot=K}
Suppose that $L$ is a self-dual Hilbert $C^*$-module and let $K\subseteq L$ be a Hilbert $C^*$-submodule with a Hilbert dual that is the kernel of a morphism $F:L\to M$, $K=\Ker F$. 
Since $L$ is self-dual, $F$ is adjointable and we require $R:=\overline{ \Im F^*}$ to be a module with Hilbert
dual.
Then $K=K^{\bot\bot}$ and we have the orthogonal decomposition:
$
L=K^{\bot\bot}\oplus K^\bot=K\oplus K^\bot.
$ 
\end{cor}

\begin{proof}
 By (\ref{eq:ker_and_coker}),   
$K=\Ker F = (\overline{ \Im F^*})^\bot=R^\bot$. By Lemma \ref{lem:botbot_ker} we have
orthogonal decompositions
$L=R^{\bot\bot} \oplus R^\bot$ and $L=K^{\bot\bot} \oplus K^\bot$. 
From the first decomposition we obtain $L=K\oplus K^\bot$. Then from the second one we obtain $K=
K^{\bot\bot}$.
\end{proof}

\section{Application to Fredholm operators}

Until now, all direct sums of submodules were orthogonal, but now we have to pay atention to the difference between Banach direct sums and Hilbert ones. Our notation $\oplus$ means the Banach direct sum, and it should be clear from the context, whether the summands are orthogonaal or not.

The following definition of Mishchenko-Fomenko $C^*$-Fredholm operator \cite{MF} plays an important role in noncommutative geometry.  
\begin{dfn}
\rm
An adjointable operator $F:M \to N$ between countably generated Hilbert $C^*$-modules over a unital
$C^*$-algebra $A$ is called $A$-\emph{Fredholm} if there exist direct sum decompositions 
${M}={M}_0 \oplus {M}_1$ and
${N}={N}_0 \oplus {N}_1$ such that 
\begin{equation}\label{summa}
F=
\left(\begin{matrix}
		F_0&0\\
		0& F_1
	\end{matrix}\right): 
{M}_0 \oplus {M}_1   \to   {N}_0 \oplus {N}_1,
\end{equation}
$F_0: {M}_0 \cong {N}_0$, $F_1:{M}_1 \to {N}_1$,
where
${M}_1$ and ${N}_1$ are projective finitely generated
modules and $M_0$ is orthogonal to $M_1$. 
Its \emph{index} is $\mathrm{index}(F)=[M_1]-[N_1]\in K_0(A)$, where $[L]$ denotes the $K$-theory class of the projective finitely generated module $L$ over $A$.
\end{dfn}

Index is stable under $A$-compact perturbations (\cite{MF}, see also \cite{MTBook}), in particular, 
$$
\mathrm{index} \left(\begin{matrix}
		F_0&0\\
		0& F_1
	\end{matrix}\right)
= \mathrm{index} \left(\begin{matrix}
		F_0&0\\
		0& 0
	\end{matrix}\right).
$$ 
For the latter operator the index can be calculated like in the classical (Hilbert space) case:
$[\Ker ] - [\mathrm{Coker}]$. But for developing of elements of Hodge theory in this setting
(see e.g. \cite{FrTroFA}) one needs to have something like this honestly (i.e. not after an $A$-compact
perturbation). The next statement solves this problem.

\begin{teo}\label{teo:Fredh}
Suppose that $M$ and $N$ are countably generated $A$-modules and $F:M\to N$ is an $A$-Fredholm operator, where $A$ is a monotone complete $C^*$-algebra.
Then $\Ker F$ and $(\Im F)^\bot$ are finitely generated projective $A$-modules and
$$
\mathrm{index}(F)=[\Ker F]-[(\Im F)^\bot]\in K_0(A).
$$	
\end{teo}

\begin{proof}
For the decomposition from the above definition, we can assume that the sum on the right-hand side of (\ref{summa}) is also orthogonal. Indeed,
since $F_0: {M}_0 \cong {N}_0$ is an adjointable isomorphism, ${N}_0$ has an orthogonal complement:
$N={N}_0  \oplus ({N}_0)^\bot$ (see e.g. \cite[Theorem 2.3.3]{MTBook}, one can write down the corresponding orthogonal projection explicitly: $P=F(F^*F)^{-1}F^*$). Then, for
\[
\widetilde{F}=
\left(\begin{matrix}
		F_0&0\\
		0& (1-P) F_1
	\end{matrix}\right): 
{M}_0 \oplus {M}_1   \to   {N}_0 \oplus ({N}_0)^\bot,
\]
we have $({N}_0)^\bot \cong  {N}_1$, so
$\mathrm{index}(F)=[M_1]-[N_1]=[M_1]-[({N}_0)^\bot]=\mathrm{index}(\widetilde{F})$,
$\Ker \widetilde{F}=\Ker F$, $\Im \widetilde{F}= \Im F$. So, we assume that ${N}={N}_0 \oplus {N}_1$ with $N_0\perp N_1$.

Let
$x=x_0+x_1$, $x_0\in {M}_0$, $x_1\in {M}_1$ and
$F(x)=0$, so $0=F_0(x_0)+F_1(x_1)\in {N}_0\oplus {N}_1$. Thus
$F_0(x_0)=0$, $F_1(x_1)=0$, so $x_0=0$ and $x\in {M}_1$. Thus
$\Ker F=\Ker F_1\subseteq {M}_1$. 

Consider $F_1:M_1 \to N_1$. Then $M_1$ is self-dual, because it is finitely generated projective module.
Since $A$ is monotone complete, all modules have Hilbert dual, in particular, so is $\Ker F_1$.
Therefore Corollary \ref{cor:K_botbot=K} gives $M_1=\Ker F_1 \oplus (\Ker F_1)^\bot$.
Denote $M_2=(\Ker F_1)^\bot$.
Also we have $\Im F = N_0 \oplus F(M_2)$.

	Note that $(\Im F)^{\bot}\subseteq N_1$ is the kernel
	of the adjoint operator: $(\Im F)^{\bot}=\Ker F^*$ (see (\ref{eq:ker_and_coker})).
	Hence $(\Im F)^{\bot}$ is an orthogonal complementable 
	finitely generated projective submodule of $N_1$.
Evidently, $F(M_2)\subseteq ((\Im F)^\bot)^\bot_{N_1}$ and
	we obtain the following form for $F$:
	$$
	\left(\begin{smallmatrix}
		F_0&0&0\\
		0& F_2 & 0\\
		0& 0 & 0
	\end{smallmatrix}\right): M_0  \oplus (M_2 \oplus \Ker F)
\to N_0 \oplus ((\Im F)^\bot)^\bot_{N_1} \oplus (\Im F)^{\bot},
	$$
where $F_2=F_1|_{M_2}$, and all summands, on both sides, are orthogonal.
By Corollary \ref{cor:self_du}, $M_2 \cong ((\Im F)^\bot)^\bot_{N_1}$
and
\begin{eqnarray*}
\mathrm{index}(F)&=&[M_1]-[N_1]\\
&=&
[M_2]+[\Ker F]-[((\Im F)^\bot)^\bot_{N_1}]-[(\Im F)^{\bot}]\\
&=&[\Ker F]-[(\Im F)^\bot]
\end{eqnarray*}
as desired.	
\end{proof}


\end{document}